\documentclass[12pt]{amsart}
\usepackage{amssymb, colordvi}
\usepackage{multirow}
\usepackage[pdftex]{graphicx,color}

\numberwithin{equation}{section}
\newtheorem{thm}{Theorem}
\numberwithin{thm}{section}
\newtheorem{prop}[thm]{Proposition}
\newtheorem{lemma}[thm]{Lemma}
\newtheorem{corollary}[thm]{Corollary}

\newtheorem{example}[thm]{Example}
\newtheorem{remark}[thm]{Remark}
\newtheorem{definition}[thm]{Definition}

\newenvironment{rem}{\begin{remark}\rm}{\end{remark}}
\newcounter{FNC}[page]
\def\fauxfootnote#1{{\addtocounter{FNC}{2}$^\fnsymbol{FNC}$%
     \let\thefootnote\relax\footnotetext{$^\fnsymbol{FNC}$#1}}}

\newcommand{\calA}{\mathcal{A}}

\newcommand{\Q}{\mathbb{Q}}
\newcommand{\R}{\mathbb{R}}
\newcommand{\Z}{\mathbb{Z}}

\newcommand{\A}{{\mathcal A}}

\newcommand{\vol}{{\rm vol}}

\newcommand{\rk}{{\rm rk}}

\newtheorem*{descartes}{Descartes' rule of signs}
\newtheorem*{problem}{Problem}

\title[Descartes' Rule of Signs for Circuits]{Descartes' Rule of Signs for Polynomial Systems supported on Circuits} 

\author{Fr\'ed\'eric Bihan}
\address{Laboratoire de Math\'ematiques\\
         Universit\'e de Savoie\\
         73376 Le Bourget-du-Lac Cedex\\
         France}

\email{Frederic.Bihan@univ-savoie.fr}
\urladdr{http://www.lama.univ-savoie.fr/~bihan/}

\author{Alicia Dickenstein}
\address{Dto.\ de Matem\'atica, FCEN, Universidad de Buenos Aires, and IMAS (UBA-CONICET), Ciudad Universitaria, Pab.\ I, 
C1428EGA Buenos Aires, Argentina}
\email{alidick@dm.uba.ar}
\urladdr{http://mate.dm.uba.ar/~alidick}
 
\thanks{AD was partially supported by UBACYT 20020100100242, CONICET PIP 20110100580, and ANPCyT PICT 2013-1110, Argentina.}  

\begin{document}

\begin{abstract}
We give a multivariate version of Descartes'~rule of signs to bound the number
of positive real roots of a system of polynomial equations in $n$ variables with $n+2$ monomials, 
in terms of the sign variation of a sequence associated both to the exponent vectors and the given coefficients.  
We show that our bound is sharp and is related to
the signature of the circuit.
\end{abstract}

\maketitle

\section{Introduction}
The following well-known rule of signs  for univariate polynomials was proposed by Ren\'e Descartes in 1637
in ``La G\'eometrie'', an appendix to his ``Discours de la M\'ethode'', see \cite[pp.\ 96--99]{struik}:

\begin{descartes}\label{pro:Descartes_original}
Given a univariate
real polynomial $f(x) = c_0 + c_1 x + \cdots + c_r x^r$,
the number of positive real roots of $f$ (counted with multiplicity) is bounded by the number of sign 
variations 
$$sgnvar(c_0, \dots, c_r)$$ 
in the ordered sequence of the coefficients of $f$. Moreover, the difference between these integer numbers is even.
\end{descartes}
The integer $sgnvar(c_0, \dots, c_r) \in \{0, \dots, r\}$ is equal to the number of sign changes 
between two consecutive elements,  after removing
all zero elements. In other words, $sgnvar(c_0, \dots, c_r)$ is the number of distinct pairs $(i,j)$ of integers, 
$0 \leq i <j \leq r$,  which satisfy $c_i \cdot c_j <0$ and $c_\ell=0$ for any integer
$\ell$ with $i<\ell<j$.

No general multivariate generalization is known for this simple rule. Itenberg and Roy  gave in 1996~\cite{ir96} 
a lower bound for any upper bound on the number of positive solutions of a sparse system of polynomial equations.
 They used a construction based on the associated 
mixed subdivisions of the Minkowski sum of the Newton polytopes of the input polynomials, 
and the signs of the coefficients of the individual polynomials at the vertices of mixed cells. 
Recently, the first multivariate generalization of Descartes' rule in case of at most one positive real root, was stated in Theorem~1.5 
of~\cite{MFRCSD13}. The main change of viewpoint in that article is that the number of positive roots of 
a square polynomial system (of $n$ polynomials in $n$ variables) is related to the signs of the maximal 
minors of the matrix of exponents and the matrix of coefficients of the system, that is, 
to the associated oriented matroids. We take this approach to 
get a multivariate version of Descartes'~rule of signs for systems supported on circuits (see 
 Theorem~\ref{thm:main} below).
 
We fix an exponent set $\A=\{w_0,w_1,\ldots,w_{n+1}\} \subset \Z^n$ of cardinality
$n+2$
and for any given real matrix $C =(c_{i,j}) \in \R^{n \times (n+2)}$ we consider the associated 
sparse multivariate polynomial system in $n$ variables $x=(x_1, \dots, x_n)$ with support $\A$:
\begin{equation}\label{E:system}
f_i(x)=\sum_{j=0}^{n+1} c_{i,j}x^{w_j} = 0 \, , \quad i=1,\ldots,n.
\end{equation}
We denote by $n_\A(C)$ the number of real positive solutions  of ~\eqref{E:system}
counted with multiplicity.

We solve the following question:  

\begin{problem}
When $n_\A(C)$ is finite, find a sharp upper bound
in terms of the number of sign variations of an associated sequence of real numbers. 
\end{problem}

Note that $n_\A(C)$ is a {\em linear} invariant of $C$, since after multiplying $C$ on the left by an invertible 
${n \times n}$ real matrix, the resulting coefficient matrix defines a system with the same solutions. 
On the other side, $n_\A(C)$ is an {\em affine} invariant of the configuration $\A$, since the number of positive solutions 
is unaltered if we multiply each equation by a fixed monomial (so, after translation of $\A$) or if we perform a
 monomial sustitution of variables by a matrix
with nonzero determinant. It is useful to consider the integer matrix
$A \in \Z^{(n+1) \times (n+2)}$:
\begin{equation}\label{eq:A}
A=\left(
\begin{array}{ccc}
1 &  \ldots &1 \\
 w_0& \ldots   & w_{n+1} 
\end{array}
\right).
\end{equation}
\noindent Thus, $n_\A(C)$ is also a {\em linear} invariant of $A$.

We make throughout the paper the following natural assumptions.
When the convex hull of $\A$ is of dimension smaller than $n$ (or equivalently, when
the rank $\rk(A)$ of $A$ is smaller than $n+1$),  system~\eqref{E:system} corresponds to an overdetermined 
system of $n$ equations in less than $n$ variables, so there are no solutions in general.
When the matrix $C$ does not have rank $n$, 
we have a system of at most $n-1$ equations in $n$ variables.
Thus, we will then assume that 
\begin{equation}\label{eq:rank}
{\rk}(A) \, = \, n+1, \quad {\rk}(C) \, = \, n.
\end{equation}

There is a basic necessary condition for the existence of at least one positive solution  
of a system of $n$ sparse polynomial equations in $n$ variables with any number of 
fixed monomials, in particular for our system~\eqref{E:system}.  
Let $C_0, \dots, C_{n+1} \in \R^n$ denote the column vectors of the coefficient matrix $C$ and call
\begin{equation}
{\mathcal C}^\circ = \R_{>0} C_0 + \dots + \R_{>0} C_{n+1}
\end{equation}
the positive cone generated by these vectors.
Given a solution $x \in \R_{>0}^n$,  the vector $(x^{w_0}, \dots, x^{w_{n+1}})$  is positive and so
the origin ${\bf 0} \in \R^n$ belongs to $\mathcal C^\circ$.
So, necessarily
\begin{equation}\label{eq:nonempty}
{\bf 0} \in {\mathcal C}^\circ.
\end{equation}
Note that condition~\eqref{eq:nonempty}, together with the hypothesis that ${\rm rk}(C) =n$ in \eqref{eq:rank} is equivalent to ${\mathcal C}^\circ =\R^n$.
We furthermore give in Proposition~\ref{P:noloss} a necessary and sufficient condition
to have $n_A(C)<\infty$. 

Our main result is Theorem~\ref{thm:main}, where we give
a bound for $n_A(C)$ in case~\eqref{eq:nonempty} is satisfied and $n_A(C)$ is finite.
 We also prove a congruence modulo $2$ as in the classical Descartes' rule of signs, see Proposition~\ref{congruence}.
We define an {\em ordering} 
on the column vectors of the matrix $C$ (Definition~\ref{def:alpha}), which
 is determined, up to total reversal, by the signs of the maximal minors of $C$ and it is unique up to reversal
 in the generic case of a {\em uniform} matrix $C$ (that is, with all nonzero minors).
It induces an ordering on the configuration $\A$, and the bound for $n_A(C)$ provided by Theorem~\ref{thm:main}
is the number of sign variations of the corresponding ordered  sequence of coefficients in any affine relation of the given exponents $\A$. The main ingredients
of our proof are a generalization of Descartes' rule of signs for vectors spaces of analytic real-valued functions and the classical Gale duality in linear algebra.
Gale duality allows us to reduce the problem of bounding $n_A(C)$ to the problem of bounding the number of roots 
of a particular univariate rational function determined by the system.

We show in Theorem~\ref{thm:mainbisbis} that the signature of $\A$ (Definition~\ref{def:signature}) gives a 
bound for $n_\A(C)$. For instance, if $\A$ consists of the vertices of a 
simplex plus one interior point, then our results imply that $n_\A(C) \le 2$, 
that is, the number of positive real roots of 
system~\eqref{E:system} cannot exceed $2$, in {\em any} dimension (see Example~\ref{ex:simplex}).
We also recover the known upper bound $n_\A(C) \le n+1$ obtained in \cite{B13}  and \cite{P-R}, 
but we show that $n_\A(C)=n+1$ could only be attained for
a particular value of the signature of $\A$ and uniform 
matrices $C$.

An important consequence of classical Descartes' rule of signs is that the number of real roots of a real 
univariate polynomial can be bounded in terms of the number of  monomials
 (with nonzero coefficient), independently of its degree. 
In the multivariate case, Khovanskii~\cite[Corollary 7]{Fewnomials} 
proved the remarkable result that the number of nondegenerate solutions in $\R^n$ of a system of $n$ 
real polynomial equations  can also be bounded 
solely in terms of the number of distinct monomials appearing in these equations. 
In contrast to Descartes' rule, Khovanskii's bound is far from sharp and
the known refinements do not depend on the signs of the coefficients or the particular
configuration of the exponents but only on their number (see \cite[Chapters 5--6]{SottileBook}). 
Khovanskii's result, as well as ours, is also valid for real configurations $\A$.

The paper is organized as follows.
We give all the necessary definitions and state precisely our main results in the next section. 
In section~\ref{sec:coro} we present some interesting corollaries of these results.
Section~\ref{sec:proofs} contains the proofs of Theorems~\ref{thm:main}
and Proposition~\ref{P:noloss}. Finally, we show the optimality of our bounds for $n_\A(C)$ in Section~\ref{sec:opt}.


\section{Orderings and statement of our main result}\label{sec:1}

Let $\A=\{w_0,w_1,\ldots,w_{n+1}\} \subset \Z^n$, $A \in \Z^{(n+1) \times(n+2)}$ 
and $C \in \R^{n \times (n+2)}$ as in the Introduction. In particular, we assume that 
Conditions~\eqref{eq:rank} are satisfied.
For any natural number $k$, denote $[k]=\{0, \dots,k-1\}$. In particular $[n+2]=\{0, \dots,n+1\}$.

\subsection{The configuration $\A$ and the matrix $A$}\label{ssec:A}
The kernel of the matrix $A$ has dimension $1$ and can 
be generated by the following vector $\lambda \in \Z^{n+2}$.
For any $\ell \in \{0, \dots, n+1\}$, consider the matrix  $A(\ell) \in \Z^{(n+1) \times (n+1)}$ 
obtained from the matrix $A$ in~\eqref{eq:A} by removing the column $\ell$
and set $\lambda_{\ell}=(-1)^{\ell+1} \det(A(\ell))$.  
The vector $\lambda$ gives an affine relation among the elements in $\A$:
\begin{equation}\label{E:affinerelation}
\sum_{\ell=0}^{n+1} {\lambda}_{\ell}\, w_{\ell}=0 \; \mbox{,} \quad \sum_{\ell=0}^{n+1} {\lambda}_{\ell}=0.
\end{equation}

Let $I \in \Z_{>0}$ be the greatest
common divisor of $\lambda_0, \dots, \lambda_{n+1}$ and set
\begin{equation}\label{eq:lambda}
\tilde{\lambda}_i \, = \, {\lambda_i}/ I, \quad i=0, \dots, n+1.
\end{equation}
So, this (nonzero) vector $\tilde{\lambda}$ gives the unique (up to sign) 
affine relation among the elements of  $\A$ with coprime
coefficients. The integer $I$ is the index of the subgroup $\Z \A$ generated by the configuration $\A$ in $\Z^n$. 
The normalized volume ${\rm vol}_{\Z}(\A)$ of the configuration $\A$ with respect to  the lattice $\Z^n$, 
is the Euclidean volume of the convex hull of $\A$ multiplied by $n!$ (so that the unit simplex has volume $1$).
It is well-known that this normalized volume ${\rm vol}_{\Z}(\A)$ bounds the number of isolated complex solutions of
system~\eqref{E:system}. Moreover, we have ${\rm vol}_{\Z}(\A)=\, \sum_{{\lambda}_i >0} \lambda_i \, = \, - \sum_{\lambda_i <0} \lambda_i$.
The number of isolated (real) positive solutions of system~\eqref{E:system} is the same if we consider $\A$ as a configuration in $\Z \A$ or in $\Z^n$. Indeed,
we may choose a basis $(v_1,\dots,v_n)$ for $\Z \A$, and see the equations~\eqref{E:system} as a system depending on new  variables $y_i=x^{v_i}$.
The normalized volume of our configuration with respect to $\Z \A$ is defined as 
\begin{equation}\label{eq:vol}
{\rm vol}_{\Z \A}(\A)\, = \, {\rm vol}_\Z(\A) /I.
\end{equation}
So when $n_A(C)$ is finite, we have the bound
\begin{equation}
n_A(C) \, \le \, {\rm vol}_{\Z\A}(\A)\, = \, \sum_{{\lambda}_i >0} \tilde{\lambda}_i \, = \, - \sum_{\lambda_i <0} \tilde{\lambda}_i.
\end{equation}

The configuration $\A$ is said to be a {\it circuit} if $\lambda_j \neq 0$ for all $j$ in
$\{0,1,\ldots, n+1\}$, 
or equivalently, if any subset of $n+1$ points in $\A$ is affinely independent 
(i.e., its convex hull has dimension $n$). In fact, we can always reduce
our problem to this case, as we now explain. In general, $\A$ is an $(n-m)$-{\em pyramid}
over a circuit of affine dimension $m$ (with $1 \le m \le n$) 
given by the subconfiguration $\A'$ of $\A$ of $m+2$ vectors with indices corresponding to nonzero $\lambda_j$.  
Without loss of generality, assume $\A'= \{w_0, \dots, w_{m+1}\}$, so that $\lambda_j\neq 0$ for $j=0, \dots, m+1$ 
and $\lambda_j=0$ for any $j > m+1$. Then, either we can find a system equivalent to~\eqref{E:system} of the form:
\[
\begin{array}{lr}
f_i(x)=\sum_{j=0}^{m+1} c'_{i,j}x^{w_j}= 0, & \, i=1,\ldots,m,\\
f_{i}(x) = \sum_{j=0}^{m+1} c'_{i,j}x^{w_j} + x^{w_{i+1}}= 0, & \, i=m+1, \dots,n,
\end{array}
\]
or system~\eqref{E:system} has either $0$ or an infinite number of solutions.
In case we can find an equivalent system as above, 
the first $m$ equations define essentially 
a square system of $m$ equations in $m$ variables. 
A positive solution of this smaller system does not necessarily extend  
to a positive solution $x \in \R^n_{> 0}$, but if such extension exists, it is unique. So,
in this case $n_\A(C)$ is bounded by $n_{\A'}(C')$, where $C' \in \R^{m \times (m+2)}$ 
is the coefficient matrix of the first $m$ equations.
Therefore, to simplify the notation, we will assume in what follows that  {$m=n$}, that is, that $\A$ is a {circuit}.
If $\A$ is not a circuit, $n$ has to be replaced by $m$ in general in the statements.

We will need the following definition.

\begin{definition}\label{def:signature}
Given a circuit $\A = \{w_0, \dots, w_{n+1}\} \subset \Z^n$, 
and a nonzero affine relation ${\lambda} \in \Z^{n+2}$ among the $w_j$.
We call $\aleph_+ = \{ j \in [n+2]\, : \, \lambda_j > 0\}$, $\aleph_- = 
\{ j \in [n+2] \, : \, \lambda_j < 0\}$,
and we denote by $a_+$ (resp. $a_-$) the cardinality of $\aleph_+$ (resp. $\aleph_-$).
The pair $\{a_+,a_-\}$ is usually called the {signature} of $\A$. It is unordered but it may
consist of a repeated value. For notational convenience, we will denote by
\begin{equation}\label{eq:signaturea}
\sigma(\A) = {\rm min} \{a_+,a_-\}.
\end{equation}
\end{definition} 
Note that as $\sum_j {\lambda}_j=0$, both $a_+, a_- \ge 1$.

\subsection{The matrix $C$}\label{ssec:C}
Given $i, j \in [n+2], i\neq j$, we call $C(i)$ the submatrix of $C$ 
with columns indexed by the indices in $[n+2]\setminus \{i\}$,
 and by $C(i,j)$ the submatrix of  $C$  with columns indexed by $[n+2] \setminus \{i,j\}$ (in the same order).
Note that we are not assuming that $i< j$, but only that they are different.

Our first lemma points out some easy consequences of Condition~\eqref{eq:nonempty}. Note that this
condition means that there is a positive vector in the kernel of the coefficient matrix. 

\begin{lemma}\label{lem:first}
Let $C \in \R^{n \times (n+2)}$ be a matrix of rank $n$ which satisfies~\eqref{eq:nonempty}.
Then, the following assertions hold
\begin{itemize}
\item[(i)] $C$ is not a {\em pyramid}, that is, ${\rk}(C(i))=n$ for any $i \in [n+2]$.
\item[(ii)] For any index $j_1 \in [n+2]$, there exists another index $j_2$ such that
${\rk}(C(j_1,j_2))=n$.
\item[(iii)] Assume $j_1\neq j_2$ and ${\rk}(C(j_1,j_2))=n$. Then, for any other index $i$,
either ${\rk}(C(j_1,i))=n$ or ${\rk}(C(j_2,i))=n$. 
\end{itemize}
\end{lemma}
\begin{proof}
Assume $\rk (C(i)) <n$ for some index $i$. This means that the column vectors $C_j, j \neq i$ lie
in a proper subspace. Therefore, we can multiply $C$ by an invertible matrix on the left,
to get a matrix $C'$ with $i$-th column equal to the $i$-th canonical vector $e_i$ while the $i$-th 
coordinate of all other columns is equal to $0$.
So, any vector in ${\rm Ker}(C')$ has its
$i$-th coodinate equal to $0$.  
But ${\rm Ker}(C) = {\rm Ker}(C')$ and this contradicts~\eqref{eq:nonempty}, proving (i).

Given an index $j_1$, as the rank of $C(j_1)$ is $n$ by (i), it has
a square submatrix of rank $n$. Any such submatrix  is of the form $C(j_1, j_2)$ for an index $j_1 \neq j_2$. 
This proves (ii).

Assume now that the matrix $C(j_1,j_2)$ has rank $n$. We can multiply $C$ on the left by 
$C(j_1,j_2)^{-1}$ and the resulting matrix $C'$ will have an {\em identity} matrix in the columns
different from $j_1, j_2$. For each $i =1, \dots, n$, the $i$-th entry of the column $C'_{j_1}$ (resp.
$C'_{j_2}$) equals $\pm \det(C'(j_2,i))$ (resp. $\pm \det(C'(j_1,i)$). 
If both $\det(C'(j_1,i)) =\det(C'(j_2,i)) =0$, we deduce that $C'$ is a pyramid, which contradicts item (i).
\end{proof}

We now translate condition~\eqref{eq:nonempty} to the {\em Gale dual} setting.
Given a full rank matrix $C \in \R^{n \times (n+c+1)}$, 
let $B \in \R^{(n+c+1) \times (c+1)}$ be a matrix whose columns generate ${\rm Ker}(C)$ and denote by
$P_0, \dots, P_{n+c} \in \R^{c+1}$ its row vectors. 
The configurations of column vectors of $C$ and row vectors of $B$ are said to
be Gale dual. The following lemma is well-known.

\begin{lemma}\label{lem:Gale}
Let $C_0, \dots, C_{n+c}$ and $P_0, \dots, P_{n+c}$ be Gale dual configurations.
Then, ${\bf 0}$ lies in the open positive cone $\mathcal C^\circ$ generated by
$C_0, \dots, C_{n+c}$ if and only if the vectors $P_0, \dots, P_{n+c}$ lie in an open
halfspace through the origin.
\end{lemma}

\begin{proof} Any vector $u$ in the kernel of $C$ is of the form $B \mu$, with $\mu$  in
$\R^{c+1}$. Thus, 
$$u_i = \langle  P_i, \mu \rangle, \quad i=0, \dots, n+c.$$
As we remarked, the condition on the columns of $C$ means that there is a positive vector $u$ in the
kernel of the matrix $C$, and  this is clearly equivalent to the fact that the row vectors $P_i$ lie
in the positive halfspace defined by $\mu$, which proves the lemma.
\end{proof}

\begin{remark}\label{rem:DeltaP}
Given a full rank matrix $C \in \R^{n \times (n+c+1)}$ which satisfies~\eqref{eq:nonempty} and $P_0, \dots, P_{n+c}$ a Gale dual of $C$.
Lemma~\ref{lem:Gale} asserts that the cone $ {\mathcal C}_P =\R_{>0} P_0 + \dots + \R_{>0} P_{n+c}$ is strictly convex. Therefore its  cone
${\mathcal C}_P^\nu$ (consisting of those vectors $\mu \in \R^{c+1}$ with  $\langle  P_i, \mu \rangle > 0$ for any $i \in [n+c]$) is a non empty full dimensional open convex cone.   

When $c=1$, 
${\mathcal C}_P^\nu \subset \R^2$ is a non empty two-dimensional open convex cone.  Then, 
  $\{ y \in \R \, : \, (1,y) \in {\mathcal C}_P^\nu\}$ and 
$\{ y \in \R \, : \, (-1,y) \in {\mathcal C}_P^\nu\}$ are open intervals in $\R$ (possibly empty or infinite) and at least one of them is non empty.  
Up to replacing the first column of  the matrix $B$ by its opposite, we will assume that
\begin{equation}\label{eq:DeltaP}
\Delta_P =  \{ y \in \R \, : \, (1,y) \in {\mathcal C}_P^\nu\} \neq \emptyset.
\end{equation}
\end{remark}

\begin{definition}\label{def:alpha}
Let $C \in \R^{n \times (n+2)}$ be a full rank matrix. An ordering of $C$ is a 
bijection $\alpha: [n+2] \to 
[n+2]$ 
which verifies that  for any choice of Gale dual vectors $P_0, \dots, P_{n+1} \in \R^2$, there exists $\varepsilon \in \{1, -1\}$ such that
\begin{equation*}
\varepsilon\, \det(P_{\alpha_i}, P_{\alpha_j}) \ge 0, \text{ for any }  i < j.
\end{equation*}
\end{definition}

We then have:

\begin{prop}\label{prop:order}
Let $C \in \R^{n \times (n+2)}$ be a full rank matrix satisfying~\eqref{eq:nonempty}.
Then, there exists an ordering $\alpha$ of $C$.
\end{prop}

\begin{proof}
Let $P'_0, \dots, P'_{n+1}$ be any choice of Gale dual of $C$.
We know by
Lemma~\ref{lem:Gale} that the planar vectors $P'_0, \dots, P'_{n+1}$ lie in
an open halfspace through the origin.
We can thus order them according to
their arguments, that is, we can find a bijection $\alpha: [n+2] \to 
[n+2]$ which verifies that 
\begin{equation}\label{eq:orderprime}
\det(P'_{\alpha_i}, P'_{\alpha_j}) \ge 0, \quad \text{ for any }  i < j.
\end{equation}
It remains to prove that one of the two sign conditions in Definition~\ref{def:alpha}
holds for any other choice
of Gale dual configuration $P_0, \dots, P_{n+1}$. But if $B$ (resp. $B'$) denotes
the $(n+2)\times 2$ matrix with rows $p_i$ (resp. $p'_i$),  the colums of $B$ (resp. $B'$)
give a basis of ${\rm Ker}(C)$. Then, there is an invertible matrix $M \in \R^{2 \times 2}$
such that $B'\cdot M = B$. Then, $B^t = M^t {B'}^t$, and thus each vector 
$P_i$ equals the image of the vector $P'_i$ by the invertible linear map with matrix
$M^t$ in the canonical bases. Then, if $\det(M^t) > 0$,
$$\det(P_{\alpha_i}, P_{\alpha_j}) \ge 0, \quad \text{ for any }  i < j,$$
and if $\det(M^t) < 0$, we have that
$$ \det(P_{\alpha_i}, P_{\alpha_j}) \le 0, \quad \text{ for any }  i < j, $$
as wanted.
\end{proof}

We can translate an ordering $\alpha$ on the 
Gale dual vectors of $C$ to sign conditions
on the maximal minors of $C$.
We first recall the following well-known
result (stated as Lemma~2.10 in~\cite{MFRCSD13}, together with several references).
Let $C\in \R^{n \times (n+2)}, B\in \R^{(n+2)\times 2}$ be full-rank matrices with ${\rm im}(B)=\ker(C)$, 
so that the column vectors of $C$ and the row vectors $P_0, \dots P_{n+1}$ of $B$ are Gale dual configurations.
Then, there exists a nonzero real number $\delta$ such that
\begin{equation}\label{eq:detM3}
\delta\det(C (j_1, j_2))  =  (-1)^{j_1+j_2} \det(P_{j_1},P_{j_2}), 
\end{equation}
for all subsets $J=\{j_1, j_2\}\subseteq[n+2]$ of cardinality $2$ such that $j_1 < j_2$.
We immediately deduce:

\begin{prop}\label{prop:certif}
Let $C \in \R^{n \times (n+2)}$ be a full rank matrix satisfying~\eqref{eq:nonempty}.
A bijection $\alpha: [n+2] \to [n+2]$ is an ordering for $C$ if
and only if there exists $\varepsilon \in \{1, -1\}$ such that
$$\varepsilon \, (-1)^{\alpha_i+\alpha_j} \cdot \det(C(\alpha_i,\alpha_j)) \cdot \frac{\alpha_j-\alpha_i}{j-i} \geq 0,$$ 
for any distinct  $i,j \in [n+2]$.
%
%
\end{prop}

\subsection{The main result}\label{ssec:main}

Let $A$ and $C$ be matrices as in the Introduction which satisfy~\eqref{eq:rank} and~\eqref{eq:nonempty}. 
We choose
a subset $K \subset [n+2]$ 
{\em maximal} (under containment) such that
 $\det(C(i,j)) \neq 0$ for any distinct $i, j \in K$. Denote by $k\ge 2$ the cardinality of $K$.
Let $\alpha$ be an ordering for $C$.
Denote by $\bar{\alpha}$ the bijection $[k] \rightarrow K$  which is deduced from $\alpha$. 
Thus, $\det(P_{\bar{\alpha}_i}, P_{\bar{\alpha}_j}) > 0$ for any $i < j$ in $[k]$, or  
$\det(P_{\bar{\alpha}_i}, P_{\bar{\alpha}_j}) < 0$ for any $i < j$ in $[k]$. 

We furthermore denote for any $j \in [k]$
\begin{equation}\label{eq:Kj}
K_j = \{\ell \in [n+2] \, : \, \ell = \bar{\alpha}_j  \text{ or } \det(C(\bar{\alpha}_j,\ell)) = 0\}.
\end{equation}
Thus, we get a partition $\sqcup_{j\in [k]} K_{j} = [n+2]$ by the maximality of $K$ and Lemma~\ref{lem:first}.

We define now the following linear forms in the signed minors $\lambda_j$ of  the matrix $A$:
\begin{equation}\label{eq:l}
\bar{\lambda}_j \, = \, \sum_{\ell \in K_{j}} \lambda_{\ell} \, = \,  \sum_{\ell \in K_{j}} (-1)^{\ell+1} \det(A(\ell)), \quad j \in [k].
\end{equation}
This set of $k$ linear forms is independent of the choice of subset $K$.
Observe also that $k=n+2$ (that is,  $K = [n+2]$) if and only if
all maximal minors of $C$ are nonzero, that is, if $C$ is {uniform}. 
The definition of $\bar{\lambda}_j, j=0,\dots,k-1,$ will be clear in display~\eqref{eq:lambdabar}.

Moreover, we define the (ordered) sequence:
\begin{equation}\label{eq:salpha}
s_\alpha \, =  \, (\bar{\lambda}_{0},\ldots,\bar{\lambda}_{{k-1}}).
\end{equation}

\begin{remark} Note that the partition $K_1, \dots, K_k$ of $[n+2]$ does {\em not} depend on the choice of $\alpha$ nor the choice of $K$. 
In fact, the subsets $K_j$ are in bijection with the equivalence classes of the following relation on the Gale dual vectors:
$P_j$ and $P_\ell$ are equivalent if they are linearly dependent (and they lie in the same ray through the origin).
The ordering $\alpha$ is unique up to reordering the elements in each $K_j$ (and up to complete reversal); in particular, 
it is unique up to reversal when $C$ is uniform. Thus, $K$ gives a choice of one representative in each class, 
picked up by the restricted map $\bar{\alpha}$, which moreover gives an ordering on the rays contaning some $P_j$ according to their angle.
So, up to complete reversal, $s_\alpha$ only depends on the matrix $C$.
\end{remark}  

Recall that ${\rm vol}_{\Z \A}(\A)$ denotes the normalized volume defined in~\eqref{eq:vol}. We are ready to state our main result:

\begin{thm}\label{thm:main}
Let $\A=\{w_0,w_1,\ldots,w_{n+1}\} \subset \Z^n$, $A \in \Z^{(n+1) \times(n+2)}$ 
and $C \in \R^{n \times (n+2)}$ satisfying~\eqref{eq:rank} and~\eqref{eq:nonempty}. Let $\alpha$ be an ordering for $C$, 
and let $K, |K| = k \le n+2$, and the linear forms $\bar{\lambda}_j, j \in [k],$ be as above.

Then, if $n_A(C)$ is finite,
\begin{equation}\label{Eq:mainDescartes}
n_A(C)  \leq sgnvar(s_\alpha),
\end{equation}
where $s_\alpha$ is the ordered sequence of linear forms in the maximal minors of $A$ defined in~\eqref{eq:salpha}.
Moreover,
\begin{equation}\label{Eq:mainDescartesvol}
n_A(C) \leq  {\rm min} \{sgnvar(s_\alpha), {\rm vol}_{\Z \A}(\A)\}.
\end{equation}
\end{thm}

At a first glance, one might think that the right member of 
\eqref{Eq:mainDescartes} depends only on the exponent vectors $\A$,
but this is not the case since the bijection $\alpha$ is determined by the coefficient matrix $C$.
This is in a sense dual to the statement of the classical Descartes' rule of signs 
for univariate polynomials with respect to a monomial basis, 
where the exponents determine the ordering, while the sign variation is about the coefficients. Interestingly enough, 
we also get a congruence modulo $2$ as in the classical Descartes' rule of signs.

\begin{example}
In case $n=1$, Theorem \ref{thm:main} is of course equivalent to the classical Descartes' rule of signs applied to 
trinomials in one variable.
Given $f(x)=c_0 x^{w_0}+c_1 x^{w_1}+c_2 x^{w_2} \in \R[x]$ a trinomial in one variable, where $w_0 < w_1 < w_2$, $n_A(C)$ is the number of positive roots of $f$ (counted with multiplicity).
In any affine relation $\lambda_0 \, w_0+\lambda_1 \, w_1+\lambda_2 \, w_2=0$,
the coefficients $\lambda_0$ and $\lambda_2$ have the same sign, which is opposite to that of $\lambda_1$. Note that always $ {\rm vol}_{\Z \A}(\A)\ge 2 = n+1$.
Assume $c_0, c_1, c_2 \neq 0$. A Gale dual of
the coefficient matrix $C$ is given by the vectors $P_0=(c_1,c_2), P_1=(-c_0,0), P_2=(0, -c_0)$. 
Considering the possible signs of the coefficients and the corresponding ordering $\alpha:[3] \to [3]$, it is straightforward to check that $sgnvar(s_\alpha) = sgnvar(c_0,c_1, c_2)$,
which is the classical Descartes' bound. The comparison is immediate if some $c_i=0$.
\end{example}

\begin{example}
Given any $n$, 
we may translate $\A$ and assume that $w_0=(0,\ldots,0)$. 
In case $\A$ has cardinality $n+2$,Theorem~1.5 in \cite{MFRCSD13} asserts the following:
If $\mbox{det} \; (A(j)) \cdot  \mbox{det} \; (C(0,j)) \geq 0$
for $j=1,\ldots,n+1$ and at least one of the inequalities $\mbox{det} \; (A(j)) \cdot  \mbox{det} \; (C(0,j)) \neq 0$ holds, then $n_A(C) \leq 1$.

We can assume that $C$ has maximal rank and satisfies condition~\ref{eq:nonempty}, since otherwise $n_A(C)=0$, 
and that $\A$ is a circuit, that is, $\lambda_j= (-1)^j \det(A(j)) \neq 0$ for all $j$. Consider the vector 
$$B^0=(0, - \det(C(0,1)), \det(C(0,2)), \dots,
(-1)^{n+1} \det(C(0,n+1))^t $$
in ${\rm ker}(C)$.
Take a choice of Gale dual $P_0, \dots, P_{n+1}$ of $C$ for which the coefficients of $B^0$ give the second coordinates of the $P_j$ 
and let $\alpha$ be an ordering of $C$. As the vectors $P_j$ lie in an open halfspace, $\alpha$ will indicate first all indices $j$ with $B^0_j < 0$, then $j=0$ and then those with
 $B^0_j >0$, if any (or just the opposite). Assume first that all pairs $P_i, P_j$ are linearly independent. 
 It follows that
for each $j \ge 1$, $B^0_j$ is nonzero and has the same sign as
the coefficient $\lambda_j$. When we consider the sign variation $sgnvar(s_\alpha)$ we will be then taking the sign variation of 
those coefficients $\lambda_j$ which are negative, and then those which are positive, if any. It follows that $sgnvar(s_\alpha) \le 1$ and we get $n_A(C) \le 1$  
via Theorem~\ref{thm:main}. When $P_j, P_\ell$ are linearly dependent, as the Gale configuration lies in an open halfspace, we deduce that $B^0_j$ has the same sign
as $B^0_\ell$. So, when we compute the linear forms $\bar{\lambda}_j$ we
add numbers with this same sign and therefore we also get that our bound~\eqref{Eq:mainDescartesvol} reads $n_A(C)\le1$.
\end{example}

In order to state a sufficient condition for the sign variation of the sequence $s_\alpha$ and $n_A(C)$ to have the same parity, 
we introduce the following property on the configuration $\A$. We say that $\A$ {\it has a Cayley structure}
if $\A$ is contained in two parallel hyperplanes. This is equivalent to the fact that there exists a nonzero vector
with $0,1$ entries other than the all-one vector $(1, \dots,1)$ which lies in the $\Q$-row span of the matrix $A$. 

\begin{prop}\label{congruence}
With our previous assumptions about $A$ and $C$ and our previous notations, assume moreover 
that both $\bar{\lambda}_0, \bar{\lambda}_{k-1}$ are different from $0$. Then, 
the difference $sgnvar(s_\alpha)-n_A(C)$ in ~\eqref{Eq:mainDescartes}  is an even 
integer number. Thus, $n_A(C) > 0$ if $sgnvar(s_\alpha)$ is odd.

In particular, the difference is even for any $C$ when $A$ does not have a Cayley structure, and it is even for any $A$ when
$C$ is uniform.
\end{prop}

We now state the following general result, which gives a necessary and sufficient condition for
$n_A(C)$ to be finite as required in the statement of Theorem~\ref{thm:main}.

\begin{prop}\label{P:noloss}
Let $A,C$ of maximal rank. Assume $n_A(C) >0$ and let
$d =(d_0, \dots, d_{n+1})$ be 
a nonzero vector in ${\rm ker}(C)$. Let $\alpha$ be an ordering of $C$.
Then $n_A(C)$ is finite if and only if
there exists $r \in [k]$ such that $d_{{\bar{\alpha}_r}} \cdot \bar{\lambda}_{r} 
\neq 0$.
\end{prop}

The proofs of Theorem~\ref{thm:main}, Proposition~\ref{congruence} and Proposition~\ref{P:noloss} will be given
in Section~\ref{sec:proofs}.

\section{Bounds and signature of the circuit}\label{sec:coro}
 Along this section, we keep the previous notations. In particular,
 we have a configuration $\A  \subset \Z^n$, $A \in \Z^{(n+1) \times(n+2)}$ 
and $C \in \R^{n \times (n+2)}$ as in the Introduction. We assume that $A, C$ satisfy
~\eqref{eq:rank} and the necessary condition~\eqref{eq:nonempty}. Let $\alpha$ be an ordering for $C$. 

The first consequence of Theorem~\ref{thm:main} is the following.

\begin{corollary}\label{cor:kminus1}
Let $k$ be the cardinality of a maximal subset $K \subset [n+2]$ with $\det(C(i,j)) \neq 0$ for any distinct $i, j \in K$.
Then
\begin{equation}\label{Eq:mainDescartesCoro}
n_A(C)  \leq  k -1. 
\end{equation}
\end{corollary}

This follows immediately from~\eqref{Eq:mainDescartes}, since $s_\alpha$ is a sequence of lenght $k$ and so 
its sign variation can be at most equal to $k-1$.

\begin{remark}\label{rmk:simplex}
In case that one of the columns of $C$ is identically zero, we have in fact a system with support on
a simplex. If, for instance, $C_0=0$, then the vector $(1, 0,\dots,0)$ lies in ${\rm Ker}(C)$, and so
any Gale dual configuration satifies that $P_0$ is linearly independent of the other $P_i$, which
all have first coordinate equal to $0$ and are thus linearly dependent. It follows that $k=2$ and we deduce
from Corollary~\ref{cor:kminus1} that if $n_A(C)$ is finite, there is at most a single positive solution.
As we are assuming that $\rk (C)=n$ and it is not a pyramid, there cannot be two zero columns.
\end{remark}

We will now give upper bounds for $n_\A(C)$ for any $C$ for which it is finite, only in terms of the combinatorics
of the points in the configuration $\A$. Interestingly, we  can also deduce from 
Theorem~\ref{thm:main} necessary conditions to achieve the bound $n_A(C) = n+1$.

\begin{thm}\label{thm:mainbisbis}
Let $\A$ and $C$ as in the statement of Theorem~\ref{thm:main} with $n_A(C) < \infty$.
If $\A$ has signature $\{a_+,a_-\}$ ,
\begin{equation}\label{eq:ineq}
n_A(C) \leq
\left\{
\begin{array}{ll}
2 \sigma(\A) & \quad \mbox{if} \quad a_+ \neq a_- \\
2 \sigma(\A)-1 & \quad \mbox{if} \quad a_+=a_-.
\end{array}
\right.
\end{equation}
Moreover, if $n_A(C) =n+1$, then $C$ is uniform
and the support $\A$ is a circuit with maximal signature 
\begin{equation}\label{eq:signature}
\{\lfloor \frac{n+2}{2} \rfloor, n+2 -
 \lfloor \frac{n+2}{2} \rfloor  \},
 \end{equation}
 that is, there are essentially  the same number of
 positive and negative coefficients $\lambda_j$.
\end{thm}

\begin{proof}
Let $\alpha$ be a bijection certifying the ordering property for $C$ and let $K$
be a maximal subset as before, with cardinal $k \le n+2$. 
By Theorem~\ref{thm:main}, we have that $n_\A(C)$ is bounded by 
$sgnvar(\bar{\lambda}_{ 0},\ldots,\bar{\lambda}_{ k-1})$.

For any $j \in [k]$, if
the value of $\bar{\lambda}_j$ is nonnegative (resp. nonpositive) then $K_{j} \cap \aleph_+ \neq
\emptyset$ (resp. $K_j \cap \aleph_-\neq
\emptyset$). Therefore, the number $n_+$ of positive (resp. $n_- $ of negative) terms in the sequence
$(\bar{\lambda}_{ 0},\ldots,\bar{\lambda}_{{{k}-1}})$ is at most $a_+$
(resp. $a_-$). So its sign variation is at most $2\min\{n_+, n_-\} \le 2 \sigma(\A)$. 
In case $a_+=a_-=\sigma(\A)$ and $\min\{n_+, n_-\} < \sigma(\A)$, clearly
$2 \min\{n_+, n_-\} < 2 \sigma(\A) -1$.  If $\min\{n_+, n_-\}= \sigma(\A)$, then $n_+=n_-=\sigma(\A)$ and then
$sgnvar(\bar{\lambda}_{ 0},\ldots,\bar{\lambda}_{ {k}-1}) \le 2 
\sigma(\A)-1$,
which concludes the proof of the inequalities. 

In case $n_\A(C) = n+1$, we deduce that $n+1 \le 2 \sigma(\A)$ or
moreover that $n+1 \le 2\sigma(\A) -1$ in case $a_+=a_-$. As $a_+ + \, a_- = n+2$, 
 it is easy to see that the signature of $\A$ is as in~\eqref{eq:signature}. In this case,
we get for $n$ even that $a_+ = a_-=\frac{n+2}2$ and so
$n+1 = 2 \frac{n+2}{2}  -1$, and for $n$ odd we get that the signature of
$\A$ equals $\{ \frac{n+1}2, \frac{n+3}2\}$  and so $n+1 = 2 \frac{n+1}2$. Thus,
we get equality in~\eqref{eq:ineq}, and so
$n_\A(C) = n+1 = sgnvar(\bar{\lambda}_{ 0},\ldots,\bar{\lambda}_{{k}-1})$, 
from which we deduce that $k=n+2$ and therefore $C$ is uniform.
\end{proof}

In the case of two variables, Theorem~\ref{thm:mainbisbis} asserts in particular that when $n_A(C)$ is finite, it is at most $3$ when the four points
in $\A$ are vertices of a quadrilateral and it is at most $2$ if one lies in the convex hull of the others. This result was independently proven by Forsg{\aa}rd in
\cite{F}, as a consequence of his study of coamoebas of fewnomials. As we stated in the Introduction, we have the following corollary for any
number of variables. 

\begin{corollary}\label{ex:simplex}.
Let $\A$ and $C$ as in the statement of Theorem~\ref{thm:main} with $n_A(C) < \infty$.
If  $\A$ consists of the vertices of a 
simplex plus one interior point  then
$n_\A(C) \le 2$.
\end{corollary}

Indeed, if a configuration $\A$ consists of the vertices of a 
simplex plus one interior point (in any dimension $n$), 
its signature equals $\{1, n+1\}$  and so
$\sigma(\A)=1$.  Theorem~\ref{thm:mainbisbis} ensures that
$n_\A(C) \le 2$..

\section{Proof of Theorem~\ref{thm:main}}\label{sec:proofs}

We first need to recall a generalization
of Descartes' rule of signs in the univariate case and apply it in our case via the notion of ordering in Section~\ref{S:Descartes}. 
Then, we complete the proof of our main Theorem~\ref{thm:main} in Section~\ref{S:proof}, which expands some basic facts
in~\cite{BBS,B07,B13}.

\subsection{A univariate generalization of Descartes' rule of signs and orderings}\label{S:Descartes}

Descartes' rule of signs concerns polynomials, that is, analytic functions of a single
variable that can be written as real linear combinations of monomials with integer
exponents. In fact, the same result holds for monomials with real exponents, which can
be evaluated on $\R_{>0}$. 
We recall the generalization of univariate Descartes' rule for
vector spaces generated by different choices of analytic functions in \cite{P-S} and 
 we then use it for our system in a Gale dual formulation.

\begin{definition} A sequence $(h_1,h_2,\ldots,h_s)$ of real valued analytic functions  defined 
on an open interval $\Delta \subset \R$ satisfies Descartes' rule of signs on $\Delta$ if for any
sequence $a=(a_1,a_2\ldots,a_s)$ of real numbers, the number of roots 
 of $a_1h_1+a_2h_2+\cdots+a_sh_s$ in $\Delta$ counted with multiplicity never exceeds $sgnvar(a)$.
\end{definition}

The classical Descartes' rule of signs asserts that monomial bases 
$(1,y,\ldots,y^{s-1})$ satisfy Descartes' rule of signs on the open interval $(0,+\infty)$.
We note that if $(h_1,h_2\ldots,h_s)$ satisfies Descartes' rule of signs 
on $\Delta$, then the same holds true for  the sequence $(h_s,h_{s-1},\ldots,h_1)$.

\smallskip

Recall that the Wronskian of $h_1,\ldots,h_s$ is the following determinant
$$
W(h_1,\ldots,h_s)= \mbox{det}
\left(
\begin{array}{cccc}
h_1 & h_2 & \ldots & h_s \\
h_1'& h_2' & \ldots & h_s'\\
\vdots & \vdots & \ldots & \vdots \\
h_1^{(s-1)} & h_2^{(s-1)} & \ldots & h_s^{(s-1)}
\end{array}
\right).$$
(The $(i,j)$ coefficient is the $(i-1)-th$ derivative of $h_j$.)
A well-known result asserts that $h_1,\ldots,h_s$ are linearly dependent if and only if their Wronskian
is identically zero.

The following result is proven in \cite{P-S}, part 5, items 87 and 90:

\begin{prop} [\cite{P-S}]\label{P:Descartes}
A sequence of functions $h_1,\ldots,h_s$ satisfies Descartes' rule of signs on $\Delta \subset \R$
if and only if for any collection of integers $1 \leq j_1 <j_2 \ldots < j_\ell \leq s$ we have
$$
W(h_{j_1},\ldots,h_{j_\ell}) \neq 0, \qquad (1)
$$
and for any collections of integers $1 \leq j_1 <j_2 \ldots < j_\ell \leq k$ 
and $1 \leq j_1' <j_2' \ldots < j_\ell' \leq k$ of the same size, we have
$$
W(h_{j_1},\ldots,h_{j_\ell}) \cdot W(h_{j_1'},\ldots,h_{j_\ell'}) >0. \qquad (2)
$$
\end{prop}

In particular, if a sequence of analytic functions $h_1,\ldots,h_s$ satisfies Descartes' rule of signs on $\Delta \subset \R$, 
then $h_1, \dots, h_s$ do not vanish, have the same sign on $\Gamma$ and no two of them are proportional.

We now apply Proposition~\ref{P:Descartes} to the Gale dual of a coefficient matrix $C \in \R^{n \times (n+2)}$ of rank $n$ 
satisfying the necessary condition~\eqref{eq:nonempty}. We keep the notations and definitions of Section~\ref{sec:1}. 
Let $C \in \R^{n \times (n+2)}$ be a full rank matrix satisfying~\eqref{eq:nonempty} and  $\alpha: [n+2] \to 
[n+2]$ an ordering of $C$; moreover, let $K$ be as in Section~\ref{ssec:main}
and let $\bar{\alpha}: [k] \rightarrow K$ be the restricted ordering.
Given a Gale dual configuration $P_0, \dots, P_{n+1}$ of $C$, let $\Delta_P $  the non-empty open set defined in Remark~\ref{rem:DeltaP}. 
Consider the associated linear functions 
\begin{equation}\label{eq:p}
p_j(y) = \langle P_j, (1,y) \rangle, 
\end{equation}
Then, $p_j$ are positive functions on $\Delta_P$ for any $j \in[n+2]$.

\begin{prop}\label{P:Key}
With the previous hypotheses and notations, the collection of nonzero rational functions
$(1/p_{ \bar{\alpha}_0},1/p_{\bar{\alpha}_1},\ldots,1/p_{\bar{\alpha}_{k-1}})$ satisfies Descartes' rule of signs on $\Delta_P$.
\end{prop}

\begin{proof}
We need to check the hypotheses of Proposition~\ref{P:Descartes}.
For simplicity, we only compute the Wronskian of $1/p_1,1/p_2,\ldots,1/p_{\ell}$.
Recall that $p_j(y)=P_{j,1}+P_{j,2}y$ and denote by $\gamma_\ell$ the product
\[\gamma_\ell \, = \, \prod_{i=1}^{\ell}  (-1)^{i-1} (i-1)! .\]
We have
$$
(1/p_j)^{(i-1)} =(-1)^{i-1} (i-1)! \cdot\frac{P_{j,2}^{i-1}}{p_j^i},
$$
and thus by the computation of a Vandermonde determinant we get
$$
\begin{array}{lll}
W(\frac{1}{p_1},\frac{1}{p_2},\ldots,\frac{1}{p_{\ell}}) & = &
\gamma_\ell \cdot \frac{1}{p_1\cdots p_{\ell}} \cdot \mbox{det}
\left( (\frac{P_{j,2}}{p_j})^{i-1}
\right)_{1 \leq i,j \leq \ell} \\
& & \\
& = & \gamma_\ell\cdot \frac{1}{p_1\cdots p_\ell} \cdot \prod_{1 \leq r_1< r_2 \leq \ell} (\frac{P_{r_2,2}}{p_{r_2}}-\frac{P_{r_1,2}}{p_{r_1}}).
\end{array}
$$
Moreover: for any $i, j \in [k], j < i$
$$
\frac{P_{\bar{\alpha}_i,2}}{p_{\bar{\alpha}_i}}-\frac{P_{\bar{\alpha}_j,2}}{p_{\bar{\alpha}_j}} = \frac{1}{p_{\bar{\alpha}_i} p_{\bar{\alpha}_j}} \cdot \mbox{det}
(P_{\bar{\alpha}_j},P_{\bar{\alpha}_i})
 > 0  \text{\, on \, }  \Delta_P
$$
(or $\frac{P_{\bar{\alpha}_i,2}}{p_{\bar{\alpha}_i}}-\frac{P_{\bar{\alpha}_j,2}}{p_{\bar{\alpha}_j}} <0$ on $\Delta_P$ for any $i, j \in [k], j < i$).
Therefore, we verify condition $(1)$ in Proposition \ref{P:Descartes} that no Wronskian of distinct functions
taken among $1/p_{\bar{\alpha}_i}$ for $i \in [k]$ vanishes,  and moreover, the sign of the Wronskian only depends on the number $\ell$ of functions we are considering.
This implies condition $(2)$  in Proposition~\ref{P:Descartes} .
\end{proof}

\subsection{The proof of Theorem~\ref{thm:main}}\label{S:proof}

We start with a basic result we will need about sign variations.

\begin{lemma} \label{L:sgnvar}
Let $s=(c_0,c_1,\ldots,c_{k-1})$ be a sequence of real numbers such that the sign
variation $sgnvar(s)$ is nonzero.
Then $sgnvar(s)$ is the minimum of the quantities $s_q, q\in [k]$, defined by  $$s_0=1+sgnvar(c_{1},\ldots,c_{k-1}),  \quad s_{k-1}=1+sgnvar(-c_0,\ldots,-c_{k-2}),$$
$$\text{and } \, s_q=1+sgnvar(-c_0, \ldots, -c_{q-1},c_{q+1},\ldots,c_{k-1}) \; , \quad q=1,\ldots,k-2.$$
\end{lemma}

\begin{proof}
For any $q \in [k]$, we easily have $sgnvar(s) \leq s_q$.
Assume  that the sign variation $sgnvar(s)$ is not the minimum of these quantities $s_q$.
Then $s_q-sgnvar(s) >0$ for any $q \in [k]$.
But $s_q-sgnvar(s)$ is equal to $1-sgnvar(c_0,c_1)$ for $q=0$, to $1-sgnvar(c_{q-1},c_q,c_{q+1})$
for $1\leq q\leq k-2$ and to $1-sgnvar(c_{k-2},c_{k-1})$ for $q=k-1$.
It follows that if $s_q-sgnvar(s) >0$ for all $q \in [k]$, then $sgnvar(s) = 0$.
\end{proof}

Let $C$ be a coefficient matrix 
of maximal rank satisfying condition~\eqref{eq:nonempty}.
As we saw in item (ii) of Lemma~\ref{lem:first},  for any given $j_1 \in [n+2]$ there exists $j_2 \in [n+2] \setminus \{j_1\}$ such 
that $\mbox{det} (C(j_1,j_2)) \neq 0$. Indeed, for any pair of distinct indices $i, j \in [k]$, $\det(C({\bar{\alpha}_i, \bar{\alpha}_j}))\neq 0$. 
We will understand that $\det(C({\bar{\alpha}_i, \bar{\alpha}_i)})=0$.
We define for each $i \in [k]$ the column vector $B^{i}$ with $\ell$-th coordinate equal to
\begin{equation}\label{eq:Bi}
B^{i}_\ell \, = \,
\left\{ \begin{array}{ll}
(-1)^{\ell} \det(C({\bar{\alpha}_i}, \ell)), \quad & \ell \in [n+2], \, { \bar{\alpha}_i} \leq \ell\\
(-1)^{\ell+1} \det(C({ \bar{\alpha}_i}, \ell)), \quad &  \ell \in [n+2], \, { \bar{\alpha}_i} > \ell.
\end{array}
\right.
\end{equation} 
Then $B^{i} \in {\rm ker}(C)$
for any $i \in [k]$, and for any pair of different indices
$i,j \in [k]$, the vectors $B^i, B^j$ span ${\rm ker}(C)$. Moreover, assume that ${\bar{\alpha}_i} < { \bar{\alpha}_j}$ and consider the basis
of $\ker(C)$ given by the vectors:
$$P^{i} = (-1)^{{\bar{\alpha}_j}} \det(C({\bar{\alpha}_i, \bar{\alpha}_j}))^{-1} B^{i},
P^{j} = (-1)^{1+{\bar{\alpha}_i}} \det(C({\bar{\alpha}_i, \bar{\alpha}_j}))^{-1} B^{j}.$$
Then, $P^{i}_{{\bar{\alpha}_i}}=P^{j}_{ \bar{\alpha}_j}=0$,  and $P^{i}_{\bar{\alpha}_j} =
P^{j}_{ \bar{\alpha}_i}=1$. If ${\bar{\alpha}_i > \bar{\alpha}_j}$ we consider the basis given by
the opposite vectors so that the same conclusions hold.

We are now ready to present the proof of our main result.

\begin{proof}[Proof of Theorem~\ref{thm:main}]
Let $\A= \{w_0, \dots, w_{n+1}\}$, $A$, $C$, $K$ and an ordering $\alpha$ of $C$ as in the statement of Theorem~\ref{thm:main}. 
We can assume that $n_A(C) > 0$, since otherwise the bound is obvious.

Fix any $i \in [k]$ and pick any other index $j \in [k]$.  Let $B^{ji} \in \R^{(n+2) \times 2}$ be the rank $2$ matrix with columns $P^{j}, P^{i}$ in this order.
The Gale dual vectors $P_0, \dots, P_n$ read in the rows of $B^{ji}$ satisfy that
$P_{{\bar{\alpha}_i}} = (1,0), \,  P_{{\bar{\alpha}_j}} = (0,1)$.  
Thus, $x \in \R^n_{>0}$ is a solution of~\eqref{E:system} if and only if there exists $\mu \in {\mathcal C}_P^\nu$ (defined in Remark~\ref{rem:DeltaP}) such that
$(x^{w_0}, \dots, x^{w_{n+1}})^t = B^{ji} \mu^t$, or equivalently: $\mu_1=x^{w_{{\bar{\alpha}_i}}}$,
$\mu_2=x^{w_{{\bar{\alpha}_j}}}$, and 
\begin{equation}\label{eq:reduced}
x^{w_\ell} =  P_{\ell,1} x^{w_{{\bar{\alpha}_i}}} +  P_{\ell ,2} x^{w_{{\bar{\alpha}_j}}}, 
\quad \text{ for all }  \ell \in [n+2].
\end{equation}
In particular, both coordinates $\mu_1, \mu_2 \neq 0$ and the open set $\Delta_P \subset \R$ (defined in~eq\ref{eq:DeltaP}) is a nonempty interval. 
Up to a translation of the configuration $\A$, we can assume without loss of generality that $w_{{\bar{\alpha}_i}}=0$. 
Then,  $\mu_1=1$. Consider the function $g : \Delta_P \to \R$ defined by
\begin{equation}\label{eq:g}
g(y) = \prod_{\ell \in [n+2]}  p_\ell(y)^{\tilde{\lambda}_\ell},
\end{equation}
where $\tilde{\lambda}_\ell$ are the coefficients of the coprime affine relation
among the $w_\ell$ defined in~\eqref{eq:lambda} and the linear functions $p_j$ are as in~\eqref{eq:p}.
Then, $x$ is a positive solution of~\eqref{E:system} if and only if $y{=x^{w_{\bar{\alpha}_j}}}\in \Delta_P$ satisfies $g(y)=1$. 
Moreover, this bijection $x \mapsto y=x^{w_{{\bar{\alpha}_j}}}$ between positive solutions
of ~\eqref{E:system} and solutions of $g(y)=1$ in $\Delta_P$ preserves the multiplicities~\cite{BS08}.
Thus, to bound $n_A(C)$ we then have to bound the number of solutions of $g=1$ over the interval $\Delta_P$ counted with multiplicities. 
As we assume that $n_A(C)$ is finite, we get that $g\not \equiv 1$.

 The logarithm $G = log(g)$ is well defined over $\Delta_P$ and for any $y \in \Delta_P$,
and $g(y) =1 $ if and only if $G(y) =\sum_{\ell \in [n+2]} {\tilde{\lambda}_\ell} \, log(p_\ell(y))=0$.
Remark that $g(y)=1$ if and only if $y$ is a root of the polynomial
\begin{equation*}
\prod_{\tilde{\lambda_\ell}  >0}  p_\ell(y)^{\tilde{\lambda}_\ell} - \prod_{\tilde{\lambda_\ell} <0}  p_\ell(y)^{-\tilde{\lambda}_\ell},
\end{equation*}
which is nonzero because $g\not \equiv 1$ and has degree bounded by  $\vol_{\Z\A}(\A)$. 
Recall the partition of $[n+2]$ defined in~\eqref{eq:Kj}. For any $r \in [k]$, an index $\ell \in K_r$ if and only 
if there exists a positive constant $\gamma_{\ell r}$ such that $P_\ell = \gamma_{\ell r} P_{{\bar{\alpha}_r}}$. 
The derivative of $G$ over $\Delta_P$ equals
\begin{equation}\label{eq:lambdabar}
G'(y) = \sum_{\ell \in [n+2]} \tilde{\lambda}_\ell \frac{P_{\ell ,2}}{p_\ell(y)} =
 \frac{1}{I} \sum_{r \in [k]} \bar{\lambda}_r \frac{P_{{\bar{\alpha}_r},2}}{p_{{\bar{\alpha}_r}}(y)}.
\end{equation}
Thus,
\[ G'(y) =0 \text{ if and only if } \sum_{r \in [k]} \frac{\bar{\lambda}_r P_{{\bar{\alpha}_r},2}} {p_{{\bar{\alpha}_r}}(y)} =0. \]
Now, we can use Proposition~\ref{P:Key} to deduce that the number of roots of
$G'$ on $\Delta_P$ counted with multiplicities is at most 
$$s_i=sgnvar(P_{{\bar{\alpha}_0}, 2}\, \bar{\lambda}_0, \dots, P_{{\bar{\alpha}_{k-1}}, 2} \, \bar{\lambda}_{k-1}).$$
Rolle's theorem leads to the bound $s_i+1$ for the number of roots of $g$ contained in $\Delta_P$ counted with multiplicities, and thus to this bound for the number $n_A(C)$.
Note that  for any $r \in [k]$,
$$P_{{\bar{\alpha}_r},2} =(-1)^{{\bar{\alpha}_r+\bar{\alpha}_j}} \delta_{i,j,r}  \, \det(C({\bar{\alpha}_r, \bar{\alpha}_i}))\det(C({\bar{\alpha}_j, \bar{\alpha}_i}))^{-1},$$
where $\delta_{i,j,r}=1$ if $({\bar{\alpha}_j-\bar{\alpha}_i)(\bar{\alpha}_r-\bar{\alpha}_i})>0$ and $\delta_{i,j,r}=-1$
otherwise.
Thus, as suggested by the notation, the number $s_i$ does not depend on $j$. So we get in fact $k$ Descartes' rule 
of signs given by the choice of $i$ in $[k]$.
Recall that $P_{{\bar{\alpha}_i}}=(1,0)$. Thus $P_{{\bar{\alpha}_r,2}}=\det(P_{{\bar{\alpha}_i}}, P_{{\bar{\alpha}_r}})$ for any $r \in [k]$. Since $\alpha$ is an ordering, we get
either $P_{{\bar{\alpha}_r,2}}<0$ for any $r < {\bar{\alpha}_i}$ and $P_{{\bar{\alpha}_r,2}}>0$ for any $r > {\bar{\alpha}_i}$, or
$P_{{\bar{\alpha}_r,2}}>0$ for any $r < {\bar{\alpha}_i}$ and $P_{{\bar{\alpha}_r,2}}<0$ for any $r > {\bar{\alpha}_i}$. Therefore, 
$$s_i=sgnvar(-\bar{\lambda}_0, \dots, -\bar{\lambda}_{i-1}, \bar{\lambda}_{i+1}, \ldots, \bar{\lambda}_{k-1}) \quad i=1,\ldots,k-2,$$
$s_0=sgnvar(\bar{\lambda}_1, \dots, \bar{\lambda}_{k-1})$ and
$s_{k-1}=sgnvar(-\bar{\lambda}_0, \dots, -\bar{\lambda}_{k-2})$.
So we are in the hypotheses of Lemma~\ref{L:sgnvar}. Then, we can combine
the $k$ obtained bounds to get our bound~\eqref{Eq:mainDescartesvol}.
\end{proof}

\begin{proof}[Proof of Proposition~\ref{congruence}]
We use the notations in the proof of Theorem~\ref{thm:main}.
We may rewrite the function $g : \Delta_P \to \R$ considered in~\eqref{eq:g} 
as 
\[g(y) = c \cdot \prod_{r \in [k]}  p_{\bar{\alpha}_r}(y)^{\bar{\lambda}_r},\]
where $c$ is some positive constant. We already saw that $n_A(C)$ is the number of solutions of $g=1$ over the interval $\Delta_P$ counted with multiplicities.

Given the ordering $\alpha$, the cone ${\mathcal C}_P$ in Remark~\ref{rem:DeltaP} equals
\begin{equation}\label{eq:cone}
{\mathcal C}_P \, = \, \R_{>0} P_{\bar{\alpha}_0} + \R_{>0} P_{\bar{\alpha}_{k-1}}.
\end{equation}
As $(1,0), (0,1)$ belong to the chosen Gale dual of $C$, the dual cone ${\mathcal C}_P^\nu$  is contained in the first quadrant.  Then,
the open interval $\Delta_P$ is bounded unless $P_{\bar{\alpha}_0} = (1,0)$, in which case $p_{\bar{\alpha}_0}$ equals the constant function $1$.
Note that the endpoints $a$ and $b$ of $\Delta_P$ are the roots of $p_{\bar{\alpha}_0}$ and $p_{\bar{\alpha}_{k-1}}$  (where $\infty$ is considered as the root of the constant $1$).
We get that $n_A(C)$ is even or odd according to whether the signs of $g-1$ for $y \in \Delta_P$ close to $a$ and $b$ respectively are the same or are different. As we are assuming
that both ${\bar{\lambda}_0}, {\bar{\lambda}_{k-1}} \neq 0$,
it follows from~\eqref{eq:cone} that the
signs of $g-1$ near $a,b$ are minus those of ${\bar{\lambda}_{k-1}}$ and ${\bar{\lambda}_0}$, so that $n_A(C)$ and  
the sign variation of the pair $sgnvar(\bar{\lambda}_0, \bar{\lambda}_{k-1})$ are congruent modulo $2$.
It remains to note that $sgnvar(\bar{\lambda}_0, \bar{\lambda}_{k-1})$ and the sign variation of the whole sequence $signvar(s_{\alpha})$ are congruent modulo $2$.

In case $\A$ does not have a Cayley structure, no proper subsum of the entries $\lambda_i$ can be zero, 
in particular we have that ${\bar{\lambda}_0}, {\bar{\lambda}_{k-1}} \neq 0$. On the other side, if $C$ is uniform,
then $k=n+2$ and each coefficient $\bar{\lambda}_j$ is equal to one of the coefficients $\lambda_i$, which we are assuming that are all non-vanishing.
\end{proof}

\begin{rem}\label{rem:noparity}
 The proof of Proposition~\ref{congruence} makes clear why the hyphoteses we made are needed. In fact, the result is not true in general if one of
 $\bar{\lambda}_0$ or $\bar{\lambda}_{k-1}$ equals $0$. Consider for example any $n \ge 3$ and $\lambda = (\lambda_0, \dots, \lambda_{n-2}, 1, -2,1)$ . Then $\lambda_0+ \cdots+ \lambda_{n-2}=0$ (since the sum of the coordinates of $\lambda$ is equal to zero), so that
 $\A$ has a Cayley structure. Let $C$ have the following Gale dual: $P_0= P_1= \dots = P_{n-2}=(1,0)$, $P_{n-1}=(3,21/8)$,
$P_n=(1,1)$ and $P_{n+1}=(0,1)$. Then, we get that $k=4$, $\bar{\alpha}(0)=0, \bar{\alpha}(1)=n-1, \bar{\alpha}(2)=n,
\bar{\alpha}(3)=n+1$, and
$\bar{\lambda}_0=0, \bar{\lambda}_1= 1, \bar{\lambda}_2=-2, \bar{\lambda}_3=1$.
Therefore the function $g$ in~\eqref{eq:g} equals
$$g = \frac{ 3 y (1+7/8 y)}{(1+y)^2}.$$
In our case, the cone ${\mathcal C}_P$ in~\eqref{eq:cone} and its dual cone equal the first quadrant and so
$\Delta_P= (0, \infty).$
For any $y>0$, $g(y)=1$ precisely when  $y (3+21/8 y)-(1+y)^2= 13/8 y^2 + y -1=0$, which  has only one positive
root, while on the other side, $signvar(\bar{\lambda}_0, \dots,\bar{\lambda}_{3})= 2$ is even.
\end{rem}

We end this section with the proof of Proposition~\ref{P:noloss}.

\begin{proof}[Proof of Proposition~\ref{P:noloss}]
Since $n_A(C) >0$, condition~\eqref{eq:nonempty} is satisfied.
We use the notations of the proof of Theorem~\ref{thm:main} above.

Assume that $d_{\bar{\alpha}_r} \cdot \bar{\lambda}_{r} 
\neq 0$ for some $r \in [k]$. Then, since $P^i$ and $P^j$ give a basis of ${\rm ker}(C)$,
either $P_{{\bar{\alpha}_r} ,1} \cdot \bar{\lambda}_{r} 
\neq 0$ or $P_{{\bar{\alpha}_r},2} \cdot \bar{\lambda}_{r} 
\neq 0$. We may assume $P_{{\bar{\alpha}_r,2}} \cdot \bar{\lambda}_{r} 
\neq 0$. Then the derivative of $G$ over $\Delta_P$ is not identically zero, thus $g$ is not identically equal to $1$ on $\Delta_P$ and $n_A(C)$ is finite.
To show the other implication, assume on the contrary that $d_{{\bar{\alpha}_r}} \cdot \bar{\lambda}_{r} 
=0$ for all $r \in [k]$. Then, the derivative of $G$ over $\Delta_P$ is identically zero, and thus $g$ is constant on $\Delta_P$.
But since $n_A(C) > 0$, there exists $y \in \Delta_P$ such that $g(y)=1$. Thus $g$ is identically equal to $1$ on $\Delta_P$, and $n_A(C)$ is infinite.
\end{proof}

\section{Optimality of the bounds}\label{sec:opt}

We prove the optimality of our bound~\eqref{Eq:mainDescartesvol} in Theorem~\ref{thm:main} in general, 
by exhibiting particular configurations $A$ and associated coefficient matrices $C$ for which the bound is attained. We keep the notations in the previous sections.

In \cite{B13}, polynomial systems supported on a circuit in $\R^n$ and having
$n+1$ positive solutions have been obtained with the help of real dessins d'enfants. In \cite{P-R},  
the main tool for constructing such systems is the generalization of the
Viro's patchworking theorem obtained in \cite{Stu}. We will recall  the construction of \cite{P-R} 
since we will only have to slightly modify it in order to prove Theorem \ref{T:optimal} for any value of the sign variation.

\begin{thm}\label{T:optimal}
Let $r$ and $n$ be any integer numbers such that $0 \leq r \leq n$ and $n>0$.
\begin{enumerate}
\item 
There exist
matrices $A,C \in \R^{n \times (n+2)}$ satisfying~\eqref{eq:rank} and~\eqref{eq:nonempty}, and an ordering $\bar{\alpha}:[k] \rightarrow K$ of $C$ such that
$$
{ sgnvar(\bar{\lambda}_{_0},\ldots,\bar{\lambda}_{{k-1}})}=1+r \quad \mbox{and} \quad n_A(C)=1+r.
$$
\item For any positive integers $a_+,a_-$  there exist matrices $A,C \in \R^{n \times (n+2)}$ satisfying~\eqref{eq:rank} and~\eqref{eq:nonempty},
such that the set $\calA \subset \R^n$ consisting of the column vectors
of $A$ except the first row of $1$'s,  has signature $\{a_+, a_-\}$ and
$$n_A(C)=
\left\{
\begin{array}{ll}
2 \sigma(\A) \quad & \mbox{if} \quad a_+ \neq a_- \\
2 \sigma(\A)-1 \quad & \mbox{if} \quad a_+=a_-.
\end{array}
\right.
$$
\end{enumerate}
\end{thm}

\begin{proof}
We first note that (2) follows easily from  (1). 
In order to prove (1),  let $n \geq 2$ be any integer. The following system is considered in \cite{P-R} :
\begin{equation}\label{E:example}
\begin{array}{ccl}
x_1x_2 & = & \epsilon+ x_1^2, \\
x_{i}x_{i+1}  & = & 1+\epsilon^{2i-3} x_1^2 \, , \quad  i=2,\ldots,n-1, \\
x_n & =  & 1+ \epsilon^{2n-3}x_1^2.
\end{array}
\end{equation}
where $\epsilon$ is a positive real parameter that will be taken small enough.
The support of this system is the circuit $\calA=\{w_0,w_1,\ldots,w_n,w_{n+1}\}$ 
with $w_0=(0,\ldots,0)$, $w_i=e_i+e_{i+1}$ for $i=1,\ldots,n-1$, $ w_n=e_n$ and $w_{n+1}=2e_1$,
where $(e_1,\ldots,e_n)$ stands for the canonical basis of $\R^n$. 
Let $C$ be the associated matrix. A choice of Gale dual of $C$ is given by
the vectors $P_0=(1,0)$, $P_1=(\epsilon,1)=\epsilon(1,\epsilon^{-1})$, 
$P_i=(1,\epsilon^{2i-3})$ for $i=2,\ldots,n$ and $P_{n+1}=(0,1)$.
These points are depicted in Figure \ref{F:Pi}.
\begin{figure}[htbp]
\begin{center}
 \resizebox{.4\textwidth}{!}{
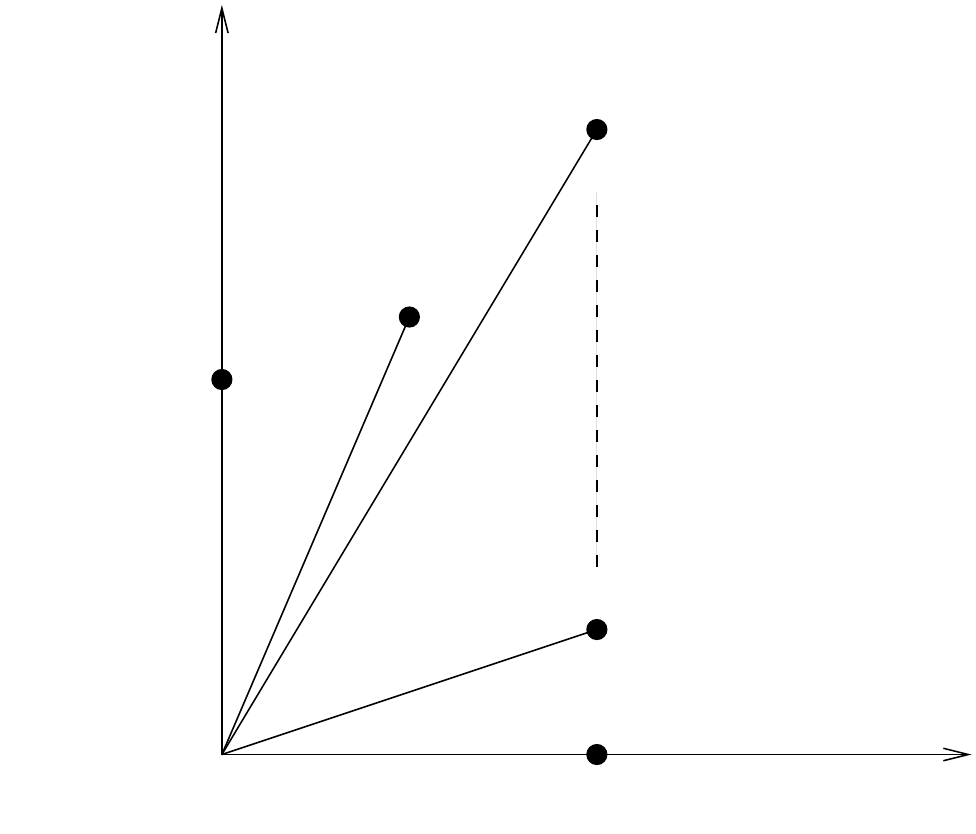
}
\caption{}
\label{F:Pi}
\end{center}
\end{figure}
It follows that  the bijection
$\alpha: [n+2] \rightarrow [n+2]$ defined by $\alpha_0=0$, 
$\alpha_1=n,\alpha_2=n-1,\ldots,\alpha_{n}=1$ and $\alpha_{n+1}=n+1$ defines
an ordering of $C$.
We compute the affine relation
\begin{equation}\label{E:relationexample} 
(-1)^n w_0+2w_1-2w_2+2w_3+\cdots+(-1)^{n-1} \cdot 2 \cdot w_n-w_{n+1}=0,
\end{equation}
which shows that $sgnvar(s_\alpha)=\vol(\A)=n+1$.
Therefore, the bound provided by Theorem \ref{thm:main} for the system~\eqref{E:example} is the maximal one $n+1$.
It is proved in \cite{P-R} that this system has precisely $n+1$ positive solutions for $\epsilon$ small enough 
(moreover, they prove that $\epsilon=1/4$ is sufficiently small).
Let us recall how the proof goes on (see \cite{P-R} for more details). Denote by $T_i$ the Newton polytope 
(a triangle) of the $i$-th equation of~\eqref{E:example} : $T_i=[0,e_i+e_{i+1}, 2e_1]$ for $i=1,\ldots,n-1$ and $T_n=[0,e_n, 
2e_1]$ (see Figure \ref{F:Ti}).
The exponents of $\epsilon$ in the system~\eqref{E:example} determine a convex mixed subdivision of the Minkowski sum
$T_1+\cdots+T_n$. The mixed cells in this mixed subdivision are $n$-dimensional zonotopes $Z_0,Z_1,\ldots,Z_{n}$. 
Each zonotope is a Minkowski sum of edges of $T_1,\ldots, T_n$.
For $i=1,\ldots,n-1$, consider the edges of $T_i$ defined by $E_{i,0}=[0,e_i+e_{i+1}]$ and $E_{i,1}=[2e_1,e_i+e_{i+1}]$. 
Consider also the edges of $T_n$ defined by $E_{n,0}=[0,e_n]$ and $E_{n,1}=[2e_1,e_{n}]$ (see Figure \ref{F:Ti}).

\begin{figure}[htbp]
\begin{center}
 \resizebox{1\textwidth}{!}{
\input{DescarTicorrected.pdf_t}
}
\caption{}
\label{F:Ti}
\end{center}
\end{figure}

For $j=1,\ldots,n-1$ the zonotope $Z_j$ is the Minkowski sum
$$Z_j=E_{1,1}+\cdots+ E_{j,1}+E_{j+1,0}+\cdots+ E_{n,0}.$$
Moreover, $Z_0=E_{1,0}+\cdots+ E_{n,0} $ and $Z_n  =  E_{1,1}+\cdots+E_{n,1}$.
To each zonotope $Z_j$ corresponds a system of binomial equations obtained from~\eqref{E:example} 
by keeping for the $i$-th equation its truncation to the $i$-th summand of $Z_j$, which is
the edge $E_{i,0}$ or $E_{i,1}$ according as $j<i$ or not. 
For instance, for $j=1,\ldots,n-1$, the binomial system corresponding to $Z_j$ is
\begin{equation}\label{E:exampleZj}
x_1x_2  =  x_1^2 \, ,   \ldots  , \, x_{j}x_{j+1}  =  \epsilon^{2j-3} x_1^2 \, , \, x_{j+1}x_{j+2}  =  1 \, ,  \ldots, \, x_n  =   1.
\end{equation}
The generalization of Viro's patchworking theorem obtained in \cite{Stu} gives that the number of positive solutions of the 
system~\eqref{E:example} is the total number of positive solutions of the binomial systems corresponding to $Z_0,\ldots,Z_n$. 
Moreover, the number of positive solutions of such a binomial system is at most one, and it is equal to one when
in each binomial equation $c_ax^a=c_bx^b$ we have $c_ac_b >0$.
Therefore, the total number of positive solutions of~\eqref{E:example} for $\epsilon >0$ small enough
can be read off the signs of the coefficients of the binomial systems corresponding to $Z_0,\ldots,Z_n$.
It follows then that~\eqref{E:example} has exactly $n+1$ positive solutions for $\epsilon$ small enough, 
which proves the optimality of item $(1)$ of Theorem \ref{T:optimal} when $r=n$.

\begin{figure}[htbp]
\begin{center}
\resizebox{.5\textwidth}{!}{
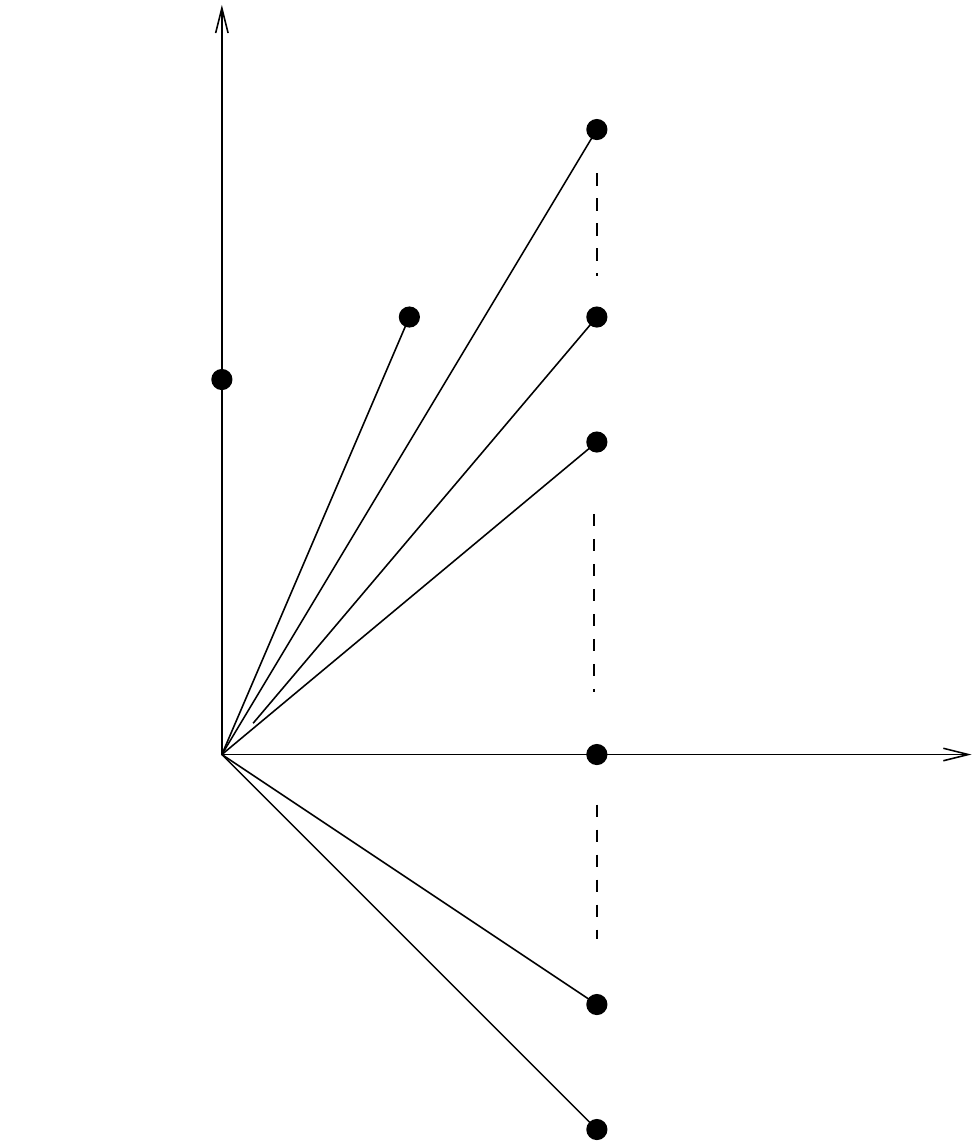
}
\caption{}
\label{F:Piallr}
\end{center}
\end{figure}

Assume now that $0 \leq r \leq n-1$.  To get a system with $1+r$ positive solutions from~\eqref{E:example}, 
we may multiply the term  $\epsilon^{2r-1} x_1^2$ of the $(r+1)$-th equation by $-1$ and keep unchanged the other equations.
So, for $r \leq n-2$ this gives the system
\begin{equation}\label{E:examplerlessn-2}
\begin{array}{ccl}
x_1x_2 & = & \epsilon+ x_1^2, \\
x_{i}x_{i+1}  & = & 1+\epsilon^{2i-3} x_1^2 \, , \quad i=2,\ldots,n-1, \, i \neq r+1, \\
x_{r+1}x_{r+2}  & = & 1\mathbf{-} \epsilon^{2r-1} x_1^2, \\
x_n & =  & 1+ \epsilon^{2n-3}x_1^2.
\end{array}
\end{equation}
while for $r=n-1$ it gives
\begin{equation}\label{E:exampler=n-1}
\begin{array}{ccl}
x_1x_2 & = & \epsilon+ x_1^2, \\
x_{i}x_{i+1}  & = & 1+\epsilon^{2i-3} x_1^2 \, ,  i=2,\ldots,n-1\\
x_n & =  & 1  {\bf -} \epsilon^{2n-3}x_1^2.
\end{array}
\end{equation}

Then, for each zonotope $Z_j$ having  the edge $E_{r+1,1}$ as a summand, the corresponding binomial
system has no positive solution. Indeed, the $(r+1)$-th equation of this binomial system is $x_{r+1}x_{r+2}=-\epsilon^{2r-1} x_1^2$ if $r \leq n-2$,
or $x_n=-\epsilon^{2n-3} x_1^2$ if $r=n-1$. On the other hand, binomial systems corresponding to zonotopes
without the edge $E_{r+1,1}$ as a summand are unchanged, and thus have still one positive solution.
The zonotopes having  $E_{r+1,1}$ as a summand are all $Z_j$  with $j \geq r+1$. Therefore, by
Viro patchworking Theorem, for $\epsilon >0$ small enough the system~\eqref{E:examplerlessn-2} has $r+1$ positive solutions 
(where $r<n-1$) while the system \eqref{E:exampler=n-1}  has $n$ positive solutions.
However, the bound $sgnvar(s_\alpha)$ given by Theorem \ref{thm:main} for these systems will be bigger than $1+r$ in general. To get a system
with $1+r$ positive solutions and for which $sgnvar(s_\alpha)=1+r$, we proceed as follows. 
First, we note that if we multiply in~\eqref{E:example}
the term $\epsilon^{2i-3} x_1^2$ of the $i$-th equation by $-1$ for $i=r+1$ and {\em any} other arbitrary values of $i \geq r+1$, 
then we still get a system with $r+1$ positive solutions for $\epsilon >0$ small enough (by the Viro patchworking Theorem). 
Indeed, for such a system, we again find that binomial systems corresponding to zonotopes $Z_j$ with $j < r+1$ have 
each a positive solution while the other binomial systems have no positive solutions.

In order to get a system with $r+1=sgnvar(s_\alpha)$ positive solutions,
we multiply in~\eqref{E:example} the term $\epsilon^{2i-3} x_1^2$ of the $i$-th equation (for $i=1$, this term is $x_1^2$) by $-1$ for any $i \geq r+1$ such that 
$i$ and $r+1$ have the same parity (and $i \leq n$). If $r \geq 1$, this gives the system

\begin{equation}\label{E:examplerfinal}
\begin{array}{ccl}
x_1x_2 & = & \epsilon+ x_1^2, \\
x_{i}x_{i+1}  & = & 1+\epsilon^{2i-3} x_1^2 \, , \quad  i=2,\ldots,r, \\
x_{i}x_{i+1}  & = & 1+{(-1)}^{i-r}\epsilon^{2i-3} x_1^2 \, , \quad  i=r+1,\ldots,n-1, \\
x_n & =  & 1+ {(-1)}^{n-r}\epsilon^{2n-3}x_1^2.
\end{array}
\end{equation}
while for $r=0$ this gives
\begin{equation}\label{E:example0final}
\begin{array}{ccl}
x_1x_2 & = & \epsilon-x_1^2, \\
x_{i}x_{i+1}  & = & 1+{(-1)}^{i}\epsilon^{2i-3} x_1^2 \, , \quad  i=2,\ldots,n-1,\\
x_n & =  & 1+ {(-1)}^{n}\epsilon^{2n-3}x_1^2.
\end{array}
\end{equation}

Figure \ref{F:Piallr} depicts a choice of Gale dual $P_0, \dots P_{n+1}$ for the
resulting matrix of coefficients $C$, from which we can read an
associated ordering  $\alpha:[n+2] \rightarrow [n+2]$, and it is straightforward to check that $sgnvar(s_\alpha)=r+1$.
\end{proof}

For a univariate real polynomial $f(x) = c_0 + c_1 x + \cdots + c_r x^r$ with any number of monomials,
the sign variation $sgnvar(c_0, \dots, c_r)$ is always bounded above by the degree of the polynomial divided by the index of the subroup of $\Z$ generated by the exponents.
This is no longer true in arbitrary dimensions, in particular in the case where the exponent set $\A$ is a circuit. 
We present two examples of particular configurations. Recall that the normalized volume $\vol_\Z(\A)$ is the multivariate  
generalization of the degree of a univariate polynomial, as it is a bound for the number of isolated complex solutions.

\begin{example} \label{ex:square}
Assume that $\vol_{\Z\A}(\A)=\vol_\Z(\A) \ge n+1$. The bound we give in~\ref{Eq:mainDescartes} reduces to
the sign variation of the ordered sequence of $\lambda_j$'s. But, for instance in
the case in which $\A$ is the configuration $\{ (0,0), (1,0), (0,1), (1,1)\} \subset \Z^2$, the normalized
volume equals $2$ (and $I=1$), while we can get matrices $C$ with $k=4$ and  with an ordering such  that the sign variation equals $3$. 
The sign variation of the corresponding
ordered sequence of $\lambda_j$'s is an upper bound, which is not sharp in this case, but
the minimum in~\eqref{Eq:mainDescartesvol} is indeed a sharp upper bound, since the two complex roots can be real positive.
\end{example}

\begin{example} \label{ex:nonsharp}
Let $n=4$ and let $\A$ be a circuit with affine relation $(1,-1,3,-3,1,-1)$. Then, 
 the maximum possible bound~\eqref{Eq:mainDescartesvol} in Theorem~\ref{thm:main} for a coefficient matrix $C$ and support $\A$ 
 is equal to $5 = \vol_\Z(\A)=\vol_{\Z\A}(\A) = n+1$. However, Theorem~0.2 in~\cite{B13} shows that there cannot be $\vol_\Z(\A)$ 
 positive real solutions to this system, so the bound in~\eqref{Eq:mainDescartesvol} is not sharp in this case.
\end{example}

We end by mentioning further work which was inspired  by the present paper. 
Theorem \ref{thm:mainbisbis} gives a necessary condition on a circuit $\A \subset \Z^n$ for the existence  of a coefficient matrix $C$ 
such that $n_A(C)$ reaches the maximal possible value $n+1$. Recently, Boulos El Hilany~\cite{H} has obtained a necessary and sufficient condition on $\A$ for this to hold.
The first author of the present paper has obtained a partial generalization of Descartes' rule of signs for any polynomial system. 
 A new bound is obtained as a sum of two terms: a sign variation of a sequence whose terms are products of maximal minors of the coefficient and 
 the exponent matrices of the system, and a quantity depending on the number of variables and the total number of monomials of the system. 
 This new bound,  which is a refinement of a result in ~\cite{BS07}, is not sharp in general.

\providecommand{\bysame}{\leavevmode\hbox to3em{\hrulefill}\thinspace}
\providecommand{\MR}{\relax\ifhmode\unskip\space\fi MR }
\providecommand{\MRhref}[2]{%
  \href{http://www.ams.org/mathscinet-getitem?mr=#1}{#2}
}
\providecommand{\href}[2]{#2}

\end{document}